\documentclass[11pt]{article}
\usepackage{amsfonts} 
\usepackage{amsmath}
\usepackage{amsthm}
\usepackage{epsfig}
\usepackage{amssymb}
\usepackage{psfrag}

\newcommand{\be}{\begin{eqnarray}}
\newcommand{\ee}{\end{eqnarray}}

\newcommand{\om}{\Omega}

\newcommand{\1}{{\bf 1}}
\newcommand{\tr}{{\rm tr}\,}
\newcommand{\cof}{{\rm cof}\,}

\newcommand{\weakstar}{\stackrel{*}{\rightharpoonup}}

\newcommand{\eps}{\epsilon}

\newcommand{\sch}{\mathcal{H}}

\newcommand{\sca}{\mathcal{A}}

\newcommand{\scl}{\mathcal{L}}

\newcommand{\et}{e_{\scriptscriptstyle{\theta}}}
\newcommand{\er}{e_{\scriptscriptstyle{R}}}

\def\Xint#1{\mathchoice
   {\XXint\displaystyle\textstyle{#1}}%
   {\XXint\textstyle\scriptstyle{#1}}%
   {\XXint\scriptstyle\scriptscriptstyle{#1}}%
   {\XXint\scriptscriptstyle\scriptscriptstyle{#1}}%
   \!\int}
\def\XXint#1#2#3{{\setbox0=\hbox{$#1{#2#3}{\int}$}
     \vcenter{\hbox{$#2#3$}}\kern-.5\wd0}}

\def\dashint{\Xint-}

\newtheorem{thm}{Theorem}[section]
\newtheorem{prop}{Proposition}[section]
\newtheorem{lemma}{Lemma}[section]

\newtheorem{defn}{Definition}[section]
\newtheorem{rem}[prop]{Remark}
\numberwithin{equation}{section}
\title{Explicit examples of Lipschitz, one-homogeneous solutions of $\log$-singular planar elliptic systems}
\author{J. Bevan 
\footnote{Department of Mathematics, University of Surrey, Guildford, Surrey,  GU2 7XH, UK.  tel: +44 (0)1483 682620. email: j.bevan@surrey.ac.uk}}
\begin{document}
\markright{One-homogeneous Lipschitz solutions}
\maketitle
\begin{abstract}\noindent  We give examples of systems of Partial Differential Equations that admit non-trivial, Lipschitz and one-homogeneous solutions in the form $u(R,\theta) = Rg(\theta)$, where $(R,\theta)$ are plane polar coordinates and $g: \mathbb{R}^{2} \to \mathbb{R}^{m}$, $m \geq 2$.   The systems are singular in the sense that they arise as the Euler-Lagrange equations of the functionals $I(u) = \int_{B}W(x,\nabla u(x))\,dx$, where
$D_{F}W(x,F)$ behaves like $\frac{1}{|x|}$ as $|x| \to 0$ and $W$ satisfies an ellipticity condition.   Such solutions cannot exist when 
\newline $|x|D_{F}W(x,F) \to 0$ as $|x| \to 0$, so the condition is optimal.   The associated analysis exploits the well-known Fefferman-Stein duality \cite{FS}.   We also discuss conditions for the uniqueness of these one-homogeneous solutions and demonstrate that they are minimizers of certain variational functionals.  
\end{abstract}

\maketitle

\section{Introduction}  

This paper exhibits explicit Lipschitz one-homogeneous maps $u:\mathbb{R}^{2} \to \mathbb{R}^{m}$ as solutions to certain systems of nonlinear Partial Differential Equations.
In terms of plane polar coordinates, such maps are of the form $u(R,\theta) = Rg(\theta)$, where $g$ is a Lipschitz function taking values in $\mathbb{R}^{m}$.  The system of nonlinear PDE is
\begin{equation}\label{el2} \partial_{x_{j}}(\partial_{F_{ij}}W(x,\nabla u(x))) = 0,  \ \ i = 1,2,\end{equation}
where the summation convention is understood; they are a componentwise form of the Euler-Lagrange equations of the integral functional  
\[ I(u) = \int_{B_{r}} W(x, \nabla u(x)) \,dx. \]
The integrand $W: B_{r} \times \mathbb{R}^{m \times 2}$ can be written 
\begin{equation}\label{w} W(x,F) = f(|F|) + \sum_{1 \leq i < j \leq m} \lambda_{ij} \ln(|x|) \, \det F^{(i,j)},\end{equation}
where all $\lambda_{ij}$ are constant,  
\[F^{(i,j)} = \left(\begin{array}{ll} F_{i1} & F_{i2} \\ F_{j1} & F_{j2} \end{array}\right)\] 
for $1 \leq i < j \leq m$, $B_{r}$ is the ball with  centre $0$ and radius $r$ in $\mathbb{R}^{2}$, and  $f:[0,\infty) \to \mathbb{R}$ is a suitably differentiable function.       We henceforth  write
\[ \gamma(F) = f(|F|). \]

The significance of such a result is twofold.  Firstly, non-trivial one-homogeneous solutions of the systems \eqref{el2} have long been sought after, and in several cases found, in the context of regularity theory, beginning with the work \cite{Ne75} of Ne\v{c}as.    Here, and in \cite{HLN},\cite{SY00}, one-homogeneous solutions are in fact minimizers of variational integrals of the form $\tilde{I}(u) = \int_{\om} W(\nabla u(x)) \,dx$ for an appropriate function $W$,  a condition implying stationarity.     See also \cite{SY02} for nonsmooth minimizers which are not one-homogeneous, but which are related to and improve upon the examples in \cite{SY00}.      The domain dimension $n$ in all these examples is at least $3$.    In contrast, Phillips showed in \cite{Ph02} that one-homogeneous stationary points of functionals like  $\tilde{I}$ with $n = 2$ are not possible:  the claim in this paper is that they are, provided we allow the integrand $W$ to depend on $x$ as well as $\nabla u(x)$.     If we do not 
insist on one-homogeneity then in two and higher dimensions \cite{MS03} and \cite{Sz04} have shown that stationary points can in general be nowhere $C^{1}$, which is an extreme form of singularity.  These solutions are constructed iteratively and as such are not explicit, an advantage which the mappings we present here do enjoy.  The price apparently to be paid for this explicitness is in the $x$-dependence of the integrands $W$ defined in \eqref{w} above.

We briefly review one-homogeneous functions and the type of singularity they can produce. By definition, a positively one-homogeneous (henceforth one-homogeneous) function $u: \mathbb{R}^{n} \to \mathbb{R}^{m}$ satisfies $u(\lambda x) = \lambda u(x)$ for all $x \in \mathbb{R}^{n}$ and all $\lambda \geq 0$, whence the representation $u(x) = Rg(\theta)$ with $g(\theta) := u(\cos \theta, \sin \theta)$ and $R=|x|$.   We also recall that a non-trivial one-homogeneous function is by definition one that is not linear.   When it exists, the weak derivative $\nabla u$ of a one-homogeneous function $u$ satisfies
\[\nabla u(x) = u(\psi(x)) \otimes \psi(x) + \nabla u(\psi(x)) - (\nabla u(\psi(x)) \psi(x)) \otimes \psi(x), \]
where $\psi(x) = \frac{x}{|x|}$.    In terms of polar coordinates, 
\[ \nabla u (R, \theta) = g(\theta) \otimes e_{\scriptscriptstyle{R}}(\theta) + g'(\theta) \otimes \et(\theta), \] 
where $\er(\theta) = (\cos \theta, \sin \theta)^{T}$ and $\et(\theta) = (-\sin \theta, \cos \theta)^{T}$.
 The gradient clearly depends only on the angular part $g(\theta)$ of $u$, so that, provided $u$ is not linear, $\nabla u$ is discontinuous at the origin.  It is in this sense that non-trivial one-homogeneous functions are singular. 

Secondly, it confirms that one of the hypotheses in the recent result \cite{Be11}[Theorem 2.1] is sharp.
We restate that result here for the reader's benefit. 
\begin{thm}[Theorem 2.1, \cite{Be11}]\label{Be11a} Let $u$ be a one-homogeneous function belonging to the class $W^{1,2}(B, \mathbb{R}^{m})$ and satisfying 
\begin{equation}\label{frog1} \int_{\om} A(x, \nabla u(x)) \cdot \nabla \varphi(x) \, dx = 0 \ \ \forall \varphi \in C_{c}^{1}(B,\mathbb{R}^{m}), \end{equation}  
where $A$ satisfies
\begin{itemize}
 \item[\textrm{(H1)}]  
$A(x,F)$ is $C^{1}$ and uniformly elliptic in the gradient argument $F$, i.e., for some fixed $\nu > 0$
\[\frac{\partial A_{ij}}{\partial F_{rs}} (x,F) a_{i}b_{j}a_{r}b_{s} \geq \nu |a|^{2}|b|^{2}\] for all $a \in \mathbb{R}^{m}$ and $b \in \mathbb{R}^{2}$; 
\item[\textrm{(H2)}] $|x|\partial_{x_{i}}A(x, F)$ is continuous on $(B\setminus\{0\}) \times \mathbb{R}^{m \times 2}$ for $i=1,2$;
\item[\textrm{(H3)}] $\lim_{|x| \to 0} |x| \partial_{x_{i}}A(x, \nabla u) = 0$ for $i=1,2$.\end{itemize}
Then $u$ is linear.
\end{thm}
The functions $W$ defined above can be chosen so that 
\[A(x,F):=D_{F}W(x,F)\] solves \eqref{frog1} with $u$ a suitable non-trivial one-homogeneous function, and such that it obeys conditions (H1) and (H2) while violating (H3).   We infer from this that condition (H3) is necessary.  See Lemma \ref{bach0} for details.

It is natural to ask whether there are circumstances under which the one-homogeneous solutions, $\bar{u}$, say, that we construct are unique.  By studying the stationarity condition \eqref{el2} in the planar case, we give a necessary and, under some additional  assumptions, sufficient condition for uniqueness.   See Propositions \ref{Bach4} and \ref{bachprop3} for details.

In view of the fact that the $\bar{u}$ solves an Euler-Lagrange equation, it is also natural to ask whether these solutions arise as minimizers of appropriate variational problems.  It turns out that they do in least two cases:
one corresponding to a problem in which functions $u$ competing in the minimization process are constrained to satisfy $\det \nabla u = 1$ a.e. (see Section \ref{con}), and another corresponding to an unconstrained problem (see Section \ref{uncon}) which has some remarkable similarities to a system constructed by  Meyers in \cite{Meyers63}.   

The paper is accordingly divided into three parts:  Section \ref{beep} gives the construction of a general class of one-homogeneous solutions to the PDE problem \eqref{el2}; Section \ref{uniqueII} considers among other things the question of uniqueness referred to above, and Section \ref{calcvar} is devoted to a variational 
interpretation of the results of Section \ref{beep}.  

\section{One-homogeneous solutions}\label{beep}

\subsection{Notation}
We denote the $m \times n$ real matrices by $\mathbb{R}^{m \times n}$, and unless stated otherwise we sum over repeated indices. 
Other standard notation includes $||\cdot ||_{k,p;\om}$ for the norm on the Sobolev space $W^{k,p}(\om)$, $||\cdot||_{p;\om}$ for the norm on $L^{p}(\om)$, and $\rightharpoonup$, $\weakstar$ to represent weak and weak$\ast$ convergence respectively in both of these spaces.   Here, $\om$ is a domain in $\mathbb{R}^{n}$.  As usual, we denote by $B(a,R)$ the ball in $\mathbb{R}^{n}$ centred at $a$ with radius $R$.  When the ball has centre zero and radius $r$ we write $B_{r}$, and when the radius is $1$ we simply write $B$ for $B_{1}$.   $\sch^{1}(\om)$ represents the Hardy space dual to $BMO(\om)$, the space of functions of Bounded Mean Oscillation (see \cite{FS,CLMS}).   In keeping with the general literature, we use  $\sch^{1}$ and $\scl^{2}$ to represent one-dimensional Hausdorff measure and two-dimensional Lebesgue measure respectively.  It will be clear from the context whether $\sch^{1}$ refers to Hardy space or $1-$ dimensional Hausdorff measure.   Unless stated otherwise, the letters a.e. refer to $\
scl^{2}-$almost everywhere.

The tensor product of two vectors $a \in \mathbb{R}^{m}$ and $b \in \mathbb{R}^{n}$ is written $a \otimes b$ and is the $m \times n$ matrix whose $(i,j)$ entry is $a_{i}b_{j}$.   The inner product of two matrices $F_{1},F_{2} \in \mathbb{R}^{m \times n}$ is $F_{1} \cdot F_{2} = \tr(F_{1}^{T}F_{2})$.   This obviously holds for vectors too.   Throughout, we use this inner product to define the norm $|F|$ on matrices $F$ via $|F|^{2} = F \cdot F$.

In plane polar coordinates $(R,\theta)$ the gradient of $\varphi: \mathbb{R}^{2} \to \mathbb{R}^{m}$ is
\begin{equation}\label{ravelI} \nabla \varphi = \varphi,_{\scriptscriptstyle{R}} \otimes \er(\theta) + \varphi,_{\tau} \otimes \et(\theta),\end{equation}
where $\er(\theta) = (\cos \theta, \sin \theta)^{T}$, $\et(\theta) = (-\sin \theta, \cos \theta)^{T}$ and 
\[\varphi,_{\tau} = \frac{1}{R}\varphi,_{\theta}.\]
We write $\varphi,_{\scriptscriptstyle{R}} =$ for the partial  derivative of $\varphi$ with respect to $R$, and similarly for $\varphi,_{\theta}$. In this notation the formula
\begin{equation}\label{det}\det \nabla \varphi = J \varphi,_{\scriptscriptstyle{R}} \cdot \varphi,_{\tau}\end{equation}
holds, where $J$ is the $2 \times 2$ matrix corresponding to a rotation of $\frac{\pi}{2}$ radians in the plane, i.e.,
\[ J = \left( \begin{array}{l l} 0 & 1 \\ -1 & 0 \end{array}\right).\]
Two useful properties of $J$ are that (i) $J^{T} = - J$, so that in particular $a \cdot Jb = - Ja \cdot b$ for any two $a,b \in \mathbb{R}^{2}$, and (ii) $\cof A = J^{T} A J$ for any $2 \times 2$ matrix $A$.    
For any set  $E$ we write $\chi_{\scriptscriptstyle{E}}$ for the characteristic (or indicator) function  of $E$.  If $E$ is an $\mathcal{L}^{2}-$measurable set and $g$ a measurable function then the integral average of $g$ over $E$ is 
\[\dashint_{E} g(x) \,dx = \frac{1}{\mathcal{L}^{2}(E)}\int_{E}g(x) \,dx.\]
A similar definition holds with $\mathcal{H}^{1}$ in place of $\mathcal{L}^{2}$.


\subsection{Construction of general one-homogeneous solutions}
 
Let $W$ be as in \eqref{w}.  

\begin{defn}(Critical or Stationary point) We say that $u$ is a critical or stationary point of the functional $I(u)=\int_{B}W(x,\nabla u(x))\,dx$, where $W$ is given by \eqref{w}, if 
\begin{equation}\label{elwk} \int_{B} D\gamma(\nabla u)\cdot \nabla \varphi + \sum_{1\leq i < j \leq m} \lambda_{ij} \ln R \ \cof (\nabla u)^{(i,j)} \cdot \nabla \varphi^{(i,j)} \, dx = 0\end{equation}
holds for all $\varphi \in C^{\infty}_{c}(B,\mathbb{R}^{m})$.\end{defn}

Thus  $u$ is a critical point of $I$ if $u$ solves the weak form of the Euler-Lagrange equations \eqref{el2} for the functional $I$.  Note that the weak form makes sense provided both $\ln(R) \, \nabla u$ and $D\gamma(\nabla u)$ belong to $L^{1}(B)$, so for now we assume that this is the case.

The aim of the next technical lemma is to rigorously convert \eqref{elwk} into the weak form of the equation
\begin{equation}\label{Luitpold}  \frac{f'(c)}{c} \bigtriangleup \!u + \Lambda u,_{\tau} = 0, \end{equation}
where $c$ is a constant such that $|\nabla u| = c$ a.e., and $\Lambda \in \mathbb{R}^{m \times m}$ is the antisymmetric matrix defined in \eqref{defLambda} below.   When $u$ is one-homogeneous this equation simplifies considerably and can be solved for any choice of the coefficients $\lambda_{ij}$: see Proposition \ref{Hechler} below.

\begin{lemma}\label{froebe}  Suppose that the function $u$ belongs to $W^{1,1}(B, \mathbb{R}^{m})$ and satisfies $|\nabla u|\ln(2+|\nabla u|)) \in L^{1}(B)$.  
\begin{itemize}\item[(a)] For $1 \leq i < j \leq m$ let 
\[ H_{ij}(F) = \det F^{(i,j)}. \]
Then
\begin{equation}\label{gudegast}\int_{B} \ln R \, DH_{ij}(\nabla u) \cdot \nabla \varphi  \,dx =  \int_{B}\left(u_{i,_{\tau}}\varphi_{j} - u_{j,_{\tau}}\varphi_{i}\right)\,\frac{dx}{R}\end{equation}
for all $\varphi \in C_{c}^{\infty}(B,\mathbb{R}^{m})$.
\item[(b)] The Euler-Lagrange equation \eqref{elwk} becomes
\[ \int_{B} D\gamma(\nabla u)\cdot \nabla \varphi -  \Lambda u_{,_{\tau}} \cdot \frac{\varphi}{R} \,dx = 0 ,\]
where the constant $m \times m$ matrix $\Lambda$ is defined by 
\begin{equation}\label{defLambda}\Lambda_{ij} = \left\{\begin{array}{l l} \lambda_{ij} & \textrm{if} \ i < j \\
0 & \textrm{if} \  i = j \\
-\lambda_{ji} & \textrm{if} \ i > j\end{array}\right.\end{equation}
\end{itemize}
\end{lemma}


\begin{proof} \textbf{(a)} We begin by establishing that it is sufficient to prove \eqref{gudegast} for smooth functions $u$.   Extending $u$ by zero outside $B$, and calling the resulting function $u$, we may suppose in particular that $\nabla u$ has compact support in $\mathbb{R}^{2}$.    By Stein's Lemma \cite[Section 5.2, p. 23]{St}, the assumption
$|\nabla u|\ln(2+|\nabla u|)) \in L^{1}(B)$ then implies that the maximal functon $M(|\nabla u|)$ belongs to $L^{1}(\mathbb{R}^{2})$, and hence that $\nabla u$ lies in $\sch^{1}(\mathbb{R}^{2})$.  This enables us to approximate $\nabla u$ using smooth gradients as follows.   Firstly, by \cite[Section 3, Corollary 1 to Theorem 6]{St}, $\sch^{1}(\mathbb{R}^{2})$ is naturally isomorphic with the Banach space $X$, where
\[  X = \{ v \in L^{1}(\mathbb{R}^{2}): \ R_{j}v \in L^{1}(\mathbb{R}^{2}), \  j = 1,2\},\]
and where $R_{j}v$ denotes the $j^{th}$ Riesz transform of $v$.  Furthermore, $X$ can be normed by 
\[ ||v||_{X} : = ||v||_{1} + ||R_{1}v||_{1}+ ||R_{2}v||_{1}.\]
Let $\rho_{\eps}$ be a standard mollifier sequence, and define $u_{\eps} = \rho_{\eps} \ast u$.   
By taking Fourier transforms, it is straightforward to prove that $R_{j}\nabla u_{\eps} =
\rho_{\eps} \ast R_{j}\nabla u$ for $j=1,2$, so that in particular by standard properties of mollified $L^{1}$ functions, $||R_{j}\nabla u_{\eps}||_{1} \to ||R_{j}\nabla u||_{1}$ as $\eps \to 0$ for $j=1,2$.   It is now evident that $\nabla u_{\eps}$ converges in the norm of $X$ to $\nabla u$, so that $\nabla u_{\eps}$ converges strongly to $\nabla u$ in $\sch^{1}(\mathbb{R}^{2})$.  

A short calculation shows that 
\[DH_{ij}(\nabla u)\,\cdot \nabla \varphi  = \cof \nabla u^{(i,j)} \cdot \nabla \varphi^{(i,j)},\]
 where $u^{(i,j)}=(u_{i},u_{j})^{T}$ for $i < j$.   Observing that $\ln R$ is a BMO function, and by appealing to the well-known Fefferman-Stein duality $(\mathcal{H}^{1})^{\ast} = BMO$, \cite[Theorem 2, p 145]{FS}, it follows that the linear functional $T_{ij}$ defined by 
 \[ T_{ij}(\nabla u) = \int_{B} \ln R \, \cof \nabla u^{(i,j)} \cdot \nabla \varphi^{(i,j)}\,dx\]
is continuous on $\sch^{1}(\mathbb{R}^{2})$ for each fixed $\varphi$.  Here we have implicitly used the fact that $\cof$ is linear on the $2 \times 2$ minors of $\nabla u \in \mathbb{R}^{m \times 2}$.   In particular, since $\nabla u_{\eps} \to \nabla u$ strongly in $\sch^{1}(\mathbb{R}^{2})$, the convergence
\[\int_{B} \ln R \, \cof \nabla u_{\eps}^{(i,j)} \cdot \nabla \varphi^{(i,j)} \,dx \to \int_{B} \ln R \, \cof \nabla u^{(i,j)} \cdot \nabla \varphi^{(i,j)} \,dx\]
as $\eps \to 0$ is immediate.    

Now consider the right-hand side of \eqref{gudegast}.   Let 
\[\tilde{\varphi}(R,\theta) = \varphi(R,\theta) - \frac{1}{2 \pi} \int_{0}^{2\pi}\varphi(R,\alpha)\,d\alpha \]
 and note that provided $u$ is $2\pi-$periodic in  $\theta$ we have
\begin{equation}\label{heavycruiser}\int_{B}\left(u_{i,_{\tau}}\varphi_{j} - u_{j,_{\theta}}\varphi_{i}\right)\,\frac{dx}{R}= \int_{B}\left(u_{i,_{\tau}}\tilde{\varphi}_{j} - u_{j,_{\tau}}\tilde{\varphi_{i}}\right)\,\frac{dx}{R}.
\end{equation}
Since $\varphi$ is smooth, it follows in particular that $\frac{\tilde{\varphi}}{R}$ has a removable singularity at the origin, and is otherwise bounded.  Writing the right-hand side of \eqref{heavycruiser} as
\[ \int_{B} \frac{\tilde{\varphi}^{(i,j)}}{R}\cdot J\nabla u^{(i,j)}\er \,dx, \]
it is therefore clear that we can pass to the limit $\eps \to 0$ in
\[ \int_{B} \frac{\tilde{\varphi}^{(i,j)}}{R}\cdot J\nabla u_{\eps}^{(i,j)}\er \,dx, \]
and replace $\tilde{\varphi}$ with $\varphi$.   In summary, it is sufficient to prove \eqref{gudegast} for smooth functions $u$.

To that end, in the following we integrate by parts and then use the fact that $\cof \nabla u$ is divergence free whenever $u$ is a smooth planar map.  
\begin{eqnarray*} \int_{B} \ln R \, DH_{ij}(\nabla u)\,\cdot \nabla \varphi \,dx & = & \int_{B} 
\ln R \, \cof \nabla u^{(i,j)} \cdot \nabla \varphi^{(i,j)}\,dx \\
& = - & \int_{B} \varphi^{(i,j)} \cdot \cof \nabla u^{(i,j)} \frac{\er}{R}\, dx \\
& = &  \int_{B} \varphi^{(i,j)} \cdot  Ju^{(i,j)}_{\tau}\,\frac{dx}{R} \\
& = & \int_{B}\left(u_{i,_{\tau}}\varphi_{j} - u_{j,_{\tau}}\varphi_{i}\right)\,\frac{dx}{R},
\end{eqnarray*}
so proving \eqref{gudegast}.

\vspace{1mm}
\textbf{(b)}  The Euler-Lagrange equation \eqref{elwk} may be written
\[\int_{B}  D\gamma(\nabla u)\cdot \nabla\varphi + \sum_{1\leq i < j \leq m} \lambda_{ij} \ln R \ DH_{ij}(\nabla u)  \cdot \nabla \varphi \, dx = 0.\]
Applying \eqref{gudegast} it follows that
\[\int_{B} D\gamma(\nabla u) \cdot \nabla\varphi\,dx + \sum_{1 \leq i < j \leq m}\int_{B}\lambda_{ij}\left(u_{i,_{\tau}}\varphi_{j} - u_{j,_{\tau}\varphi_{i}}\right)\,\frac{dx}{R} = 0.\]
Fix an index $k$ in $\{1,\ldots,m\}$ and note that the terms involving $\varphi_{k}$ in the second of these integrals are
\[ - \int_{B} \sum_{k < j \leq m}\lambda_{kj} u_{j,_{\tau}} \varphi_{k} \,\frac{dx}{R} + 
\int_{B} \sum_{1 \leq i < k}\lambda_{ik} u_{i,_{\tau}}\varphi_{k} \,\frac{dx}{R},\]
which, in terms of the matrix $\Lambda$ defined in \eqref{defLambda}, equals 
$\int_{B} -\Lambda u_{,_{\tau}}\cdot \frac{\varphi}{R}\,dx$, as required.
\end{proof}



\begin{prop}\label{Hechler}
Let the $m \times m$ matrix $\Lambda$ defined by \eqref{defLambda} in the statement of Lemma \ref{froebe} be non-zero.    Then 
\begin{itemize}\item[(a)]there exist one-homogeneous solutions $u(R,\theta)=Rg(\theta)$ to
\begin{equation}\label{kabalev} \int_{B} D\gamma(\nabla u) \cdot \nabla \varphi -  \Lambda u,_{\tau} \cdot \frac{\varphi}{R} \,dx = 0,\end{equation}
for all $\varphi \in C_{c}^{\infty}(B, \mathbb{R}^{m})$ such that
$g$ obeys the conservation law \begin{equation}\label{c1} |g|^{2}+|g'|^{2} = c^2
\end{equation}
for some constant $c$, and such that
 \begin{itemize}
\item[(i)] $u$ is linear if $m$ is an odd integer, or if $m$ is even and $\Lambda$ has a non-trivial kernel;
 
\item[(ii)] $u$ is non-linear and the set $u(B)$ is homeomorphic to a two dimensional disk in $\mathbb{R}^{m}$ covered $k$ times, $k \in \mathbb{N} \setminus \{1\}$, provided $f$ is such that 
\[ \left(\frac{(k^2-1)f'(t)}{kt}\right)^{2}\]  
lies in the (non-zero) spectrum of $-\Lambda^{2}$ for some $t > 0$.  
\end{itemize}

\item[(b)] Let $f$ belong to $C^{2}(\mathbb{R}^{+},\mathbb{R})$ and suppose that $f''(t) >0$ for all $t > 0$, $f'(0+) \geq 0$.    Then any one-homogeneous solution $u=Rg$ to \eqref{kabalev} such that $g$ is of class $C^{2}(\mathbb{S}^{1},\mathbb{R}^{m})$ satisfies the conservation law \eqref{c1} for some constant $c$.  In particular, $g$ satisfies
\[\frac{f'(c)}{c}(g''+g) + \Lambda g' = 0\]
on $[0,2 \pi]$.
\end{itemize}
\end{prop}

\begin{proof} \textbf{(a)}  Let $g(\theta) = x \cos(k \theta) + y \sin(k \theta)$ for fixed vectors $x,y \in \mathbb{R}^{m}$ and non-zero integer $k$ to be determined as follows.  Note first that \eqref{c1} holds with $c^2 = (1+k^2)|x|^{2}$ provded $|x|^2 = |y|^2$ or $k^2=1$.   In either case,  \eqref{c1} implies $|\nabla u| =c$, so that \eqref{kabalev} becomes
\[ \int_{B} \left(\frac{f'(c)}{c}(g''+g) + \Lambda g'\right)\cdot \frac{\varphi}{R} \,dx = 0\]
for all $\varphi \in C_{c}^{\infty}(B,\mathbb{R}^{m})$.
Therefore in order to solve \eqref{kabalev}, and hence \eqref{elwk}, it is sufficient to ensure that $g$ satisfies 
\begin{equation}\label{s1}\frac{f'(c)}{c}(g''+g) + \Lambda g' = 0.\end{equation} 
We consider two cases, the first of which corresponds to finding a linear solution of \eqref{kabalev}.   

\vspace{1mm}

\noindent \textbf{Case (a)(i)} If $k^2=1$ then clearly $g''+g=0$, and \eqref{s1} holds only if $\Lambda x$ and $\Lambda y$ both vanish.   If $m$ is odd then the skew-symmetry of $\Lambda$ guarantees the existence of $x$ (and hence of $y$ by taking $y=x$), while if $m$ is even solutions corresponding to $k^2=1$ exist only if $\ker \Lambda$ is non-zero. 

\vspace{1mm}
\noindent
\textbf{Case (a)(ii)} When $k \neq 1$ equation \eqref{s1} holds only if 
\begin{eqnarray}\label{k} \rho x +  \Lambda y & = & 0 \\
 \nonumber \rho y -  \Lambda x & = & 0,
\end{eqnarray}
where $\rho:=\frac{f'(c)(1-k^2)}{kc}$.   Necessarily, $\Lambda^{2} x = -\rho^{2} x$.  Note that if this can be solved for $\rho \neq 0$ then defining $y$ by the second of the above equations, namely $y = \frac{1}{\rho} \Lambda x$,  yields a solution to the first.   Moreover, $y$ so defined automatically satisfies $|y| = |x|$, which is needed in order that \eqref{c1} holds.  Therefore it is sufficient to find $x$ such that $\Lambda x = -\rho^{2}x$ for some non-zero $\rho$.    To this end, observe that the matrix $-\Lambda^{2}$ is positive semi-definite and symmetric, and so has a diagonal representation $\textrm{Diag}({\rho_{1}}^2, \ldots, {\rho_{m}}^2)$ in terms of an appropriate basis, with $\rho_{j}$ real for $1 \leq j \leq m$.  Therefore there exists a non-zero eigenvalue for $\Lambda^{2}$ provided  $\Lambda^{2} \neq 0$.   But $\Lambda^2=0$ only if all its diagonal entries vanish, and since each such entry is the Euclidean norm of a row (equivalently column) of $\Lambda$ it must be that $\Lambda^2 = 0$ is 
possible only when $\Lambda = 0$, contradicting our hypothesis.    Finally, we suppose that the integer $k$ is such that the condition on $f$ stated in (a)(ii) above holds, so that there is a non-zero eigenvalue ${\rho_{0}}^{2}$  of $\Lambda^{2}$ such that
\begin{equation}\label{whist} \frac{(k^2-1)f'(t)}{kt}= \rho_{0}\end{equation}
for some $t > 0$.  Choose $x$ so that $|x| = t(k^2+1)^{-\frac{1}{2}}$.  Recalling that $c^{2} = (1+k^2)|x|^2$ we see that this choice implies $t=c$, and hence from \eqref{whist}
\[ \frac{(k^2-1)f'(c)}{kc}= \rho_{0}.\]
Fromt this it follows that equations \eqref{k} hold with $\rho = \rho_{0}$ and $x,y$ as described above.

\vspace{1mm}
\noindent \textbf{(b)} Let $P(\theta) = |g|^{2}+|g'|^2$ and suppose that \eqref{kabalev} holds.   Let 
\[ z = \frac{f'(\sqrt{P})}{\sqrt{P}} \]
when $P > 0$, and  set $z = 0$ when $P=0$.  The assumptions on $f$ made in statement (b) above are then such that $z$ vanishes if and only if $P$ vanishes.  

Integrating by parts in \eqref{kabalev} we obtain
\[\int_{B} ( (zg')'+zg+\Lambda g') \cdot \varphi \,dR\,d\theta = 0\]
for all $\varphi \in C_{c}^{\infty}(B, \mathbb{R}^m)$, and 
the regularity assumptions on $g$ then imply that 
\begin{equation}\label{humour} (zg')'+zg +\Lambda g' = 0 \end{equation}
at all $\theta$ such that $P(\theta) > 0$.     Now let $N=\{\theta \in [0,2\pi]: \ P(\theta) > 0\}$.   We shall show in the following that $N$ is either empty, in which case \eqref{c1} holds trivially, or $N=[0,2 \pi]$.

First note that since $\Lambda$ is skew, it follows that $g' \cdot \Lambda g' = 0$, and hence from \eqref{humour} that
\[ z(g'' \cdot g' + g' \cdot g) + z'|g'|^2 = 0\]
for $\theta \in N$.   Furthermore,  $z(\theta) >  0$ for $\theta \in N$, so that the latter equation implies in particular that 
\begin{equation*} \frac{1}{2}zP' + z' |g'|^2 = 0 \end{equation*}
on $N$.   On using the definition of $z$ given above, we obtain 
\begin{equation}\label{delibes} \frac{1}{2}P' \left(z(\sqrt{P}) + \frac{\dot{z}(\sqrt{P}) |g'|^{2}}{\sqrt{P}} \right)  = 0 \end{equation}
on $N$.   

We claim that \eqref{delibes} implies that $P'=0$ on $N$.    Suppose for a contradiction that $P'(\theta) \neq 0$ for some $\theta \in N$.  Then, by \eqref{delibes},
\[ z(\sqrt{P}) +  \frac{\dot{z}(\sqrt{P}) |g'|^{2}}{\sqrt{P}} = 0,\]
which holds only if both terms are non-zero, and in particular when $\dot{z}(\sqrt{P}) <0$.    But then 
\begin{eqnarray*} z(\sqrt{P})  & = &  |\dot{z}(\sqrt{P})|\frac{|g'|^2}{\sqrt{P}} \\
& \leq & |\dot{z}(\sqrt{P})|\sqrt{P} \\
 & = & - f''(\sqrt{P}) + \frac{f'(\sqrt{P})}{\sqrt{P}}, \end{eqnarray*}
 which, when the definition of $z$ is recalled, gives $f''(\sqrt{P}) \leq 0$, contradicting the hypothesis on $f$.    It follows that $P'=0$ on the set $N$, and since $P$ is continuous it must 
be that if $N$ is non-empty then it covers all of $[0,2 \pi]$.    Clearly \eqref{c1} then holds, which concludes the proof of part (b) of the proposition.
\end{proof}

\section{Properties of and conditions for uniqueness of critical points }\label{uniqueII}

In this section we restrict attention to the planar case and consider the functionals
\[G(u) = \int_{B_{r}} \gamma(\nabla u) + \lambda \ln R \, \det \nabla u \,dx\]
where $B_{r} \subset \mathbb{R}^{2}$ and $u$ belongs to the class
\[ \sca_{p} = \{ u \in W^{1,p}(B_{r};\mathbb{R}^{2}): \  u = \bar{u} \ \textrm{on} \ \partial B_{r}, \  G(u) \in (\infty, + \infty)\}.\] 
In addition to the properties of $f$ assumed in previous sections of the paper, we suppose that $f$ obeys a polynomial growth condition of order $p$, i.e.,
\[ c_{1}|A|^{p} \leq f(|A|) \leq c_{2}(1+|A|^{p}) \ \forall \ A \in \mathbb{R}^{2}\]
for fixed positive constants $c_{1},c_{2}$.

We use Proposition \ref{Hechler} to demonstrate that such functionals possess one-homogeneous critical points, henceforth referred to as $\bar{u}$.    That $\bar{u}$ is itself a critical point of $G$ then implies various facts about other, suitably regular critical points of the functional $G$, should they exist, in the class $\sca_{p}$ for appropriate $p > 1$.   For further details see Propositions \ref{Bach4} and  \ref{bachprop3} below.   In particular, we give a geometric condition which when satisfied implies the uniqueness of the one-homogeneous critical point $\bar{u}$ of $G$ referred to above.   The strict monotonicity of the gradient of 
the integrand $\gamma$ plays an important role in the calculations and can be inferred from relatively mild assumptions on $f$: see Lemma \ref{bach0} for details



We begin by finding $\bar{u}$ under the assumptions that $f$ is $C^{2}$ on $\mathbb{R}^{+}$ and that $f''(t) > 0$ for all $t > 0$.     Let us assume that the angular part $g(\theta)$ of $\bar{u}$ is smooth enough to apply part (b) of Proposition \ref{Hechler}. Then $|\nabla \bar{u}|^{2} = c^{2}$ for some constant $c$, and 
\begin{equation}\label{lam} \Lambda  =  \left (\begin{array}{l l} 0 & \lambda \\ -\lambda & 0 \end{array}\right).\end{equation}
According to part (a)(i) of Proposition \ref{Hechler}, linear solutions exist only if $\Lambda$ has a non-trivial kernel, which it plainly does not.   Hence we suppose there is an integer $k > 1$ and $t > 0$ such that 
\[ \left(\frac{(k^2-1)f'(t)}{kt}\right)^2  = \lambda^2 \]
as stated in part (a)(ii) of the Proposition.  In these circumstances, 
\[g(\theta) = x \cos k \theta + y \sin k \theta\]
where $-\Lambda^2 x = \lambda^2 x$ and $y = \frac{1}{\lambda} \Lambda x$.   But $-\Lambda^2 = \lambda^{2} \1$, so that the choice of $\frac{x}{|x|}$ is free.   We may therefore let $x=|x|e_{1}$ and $y=|x|e_{2}$, where 
\[ |x| = t (1+k^2)^{-\frac{1}{2}},\]
and hence
\[ g(\theta) =  t (1+k^2)^{-\frac{1}{2}} e(k\theta)\] 
for all $\theta$.  The resulting one-homogeneous map is thus proportional to the k-covering map, as is summarised below:

\begin{prop}\label{freewill}Let $f$ be $C^{2}$ on $\mathbb{R}^{+}$ and such that $f''(t) > 0$ for all $t > 0$.  Suppose that either $f$ or $\lambda$ is chosen so that 
\[ \frac{(k^2-1)f'(t)}{kt} = \lambda\]
holds for some $t > 0$. Then the k-covering map
\[ \bar{u}(R,\theta) = t(1+k^2)^{-\frac{1}{2}} R e(k \theta) \] 
 is a stationary point of the functional
 \[ G(u) = \int_{B_{r}} \gamma(\nabla u) + \lambda \ln R \, \det \nabla u \,dx \]
for each $r > 0$.
\end{prop}

Henceforth we let $\bar{u}(R,\theta) = aR\er(k\theta)$ be the one-homogeneous stationary point of $G$ found in Proposition \ref{freewill} above, where the constant $a = t(1+k^2)^{-\frac{1}{2}}$.  

\begin{lemma}\label{bach0} Let $f$ be $C^{2}$ on $\mathbb{R}^{+}$ and such that $f''(t) > 0$ for all $t > 0$.  Let $\gamma(F) = f(|F|)$  be such that $D\gamma(0)$ exists and $\gamma(F) \geq \gamma(0) = 0$ for all $F \in \mathbb{R}^{2 \times 2}$.
Then 
\begin{itemize}
\item[(i)]$\gamma$ is strongly convex, $D\gamma$ is strictly monotone and continuous, and 
\item[(ii)] if, in addition, $\gamma(F) \geq \nu |F|^{2}$ for some constant $\nu > 0$ and all $F$ then the functional 
\[ W(x,F) = \gamma(F) + \lambda \ln R \, \det F\]
is such that 
\[A(x,F):=D_{\scriptscriptstyle{F}}W(x,F)\]
satisfies hypotheses (H1), (H2) but not (H3) of Theorem \ref{Be11a}.  System 
\eqref{frog1} is solved by the one-homogeneous function $\bar{u}$ defined in Proposition \ref{freewill} above.
\end{itemize} 
\end{lemma}
\begin{proof}  To prove (i), we begin by noting that the hypotheses together with  standard results from  convex analysis imply that $D\gamma(0)=0$, $f'(0)=0$ and that  $f$ is strictly convex on $(0,\infty)$.   In  particular, $f(t)>0$ and $f'(t)>0$ for $t >0$, and $f'$ is continuous on $[0,\infty)$. 
Now,  for any $2 \times 2$ matrix $\Pi$ and any non-zero $F$,
\[D^{2}\gamma(F)[\Pi,\Pi] = f''(|F|)(\Pi \cdot \hat{F})^{2} + \frac{f'(|F|)}{|F|}\left(|\Pi|^{2} -(\Pi \cdot \hat{F})^{2}\right),\]
where $\hat{F} = \frac{F}{|F|}$.   Thus $D^{2}\gamma(F) >0$ as a quadratic form,  and hence every non-zero $F$ is a point of strict convexity of $\gamma$, from which we deduce that
\[\gamma(F') - \gamma(F) > D\gamma(F)\cdot(F'-F) \ \forall \ F,F' \in  \mathbb{R}^{2 \times 2}. \]  
(This is standard: see e.g. \cite[Theorem 4.5]{Rockafellar} or \cite[Theorem 2.5.2]{Da08}.)
It remains to show that the same strict inequality holds for all $F'$ when $F=0$, which amounts to showing $\gamma(F') > 0$.  But this follows easily from the properties  of $f$ deduced above.  Thus $D\gamma$ is continuous and strictly monotone, and part (i) is proved.    

The extra hypothesis supplied in part (ii) ensures that  
\[D^{2}\gamma(F)[\Pi,\Pi] \geq \nu |\Pi|^{2}\]
for all $F$.     To confirm hypotheses (H1) we must check that 
\[ D^{2}W(x,F)[a \otimes b, a \otimes b] \geq \nu|a \otimes b|^{2}\]
independently of $x$.   But 
\begin{eqnarray*} D^{2}W(x,F)[a \otimes b, a \otimes b] & = &  D^{2}\gamma(F)[a \otimes b, a \otimes b] + 2 \lambda \ln  R \, \det(a \otimes b) \\
& \geq & \nu |a \otimes b|^{2},\end{eqnarray*}
and so  (H1) holds.   (H2) holds because 
\[ R \partial_{x_{i}} D_{F}W(x,F) = \lambda \cof F\]
is continuous everywhere.   Hypothesis (H3), however is violated:  if it were to  hold we would require
\[ \lim_{R \to 0} R \partial_{x_{i}} D_{F}W(x,F) = 0.\]
But by the calculation  above  this is false whenever $F \neq 0$.

\end{proof}
The next result will be used in connection with Lemma \ref{bach2a}. 

\begin{lemma}\label{regina} Let $\bar{u}$ be as above and assume that $u$ is a critical point of $G$ that is $C^{1}$ in a semi-open annulus $\{ x \in B_{r}: \ r - \delta < |x| \leq r\}$ for some $\delta > 0$.   Then 
\[ \nabla u(x) = \nabla \bar{u}(x) \ \ x \in \partial B_{r}.\]
\end{lemma}
\begin{proof} An approximation argument using the regularity assumption on $u$ implies that 
\[ \int_{B_{r}} D\gamma(\nabla u) \cdot \nabla\varphi + \lambda \ln R \, \cof \nabla u \cdot \nabla \varphi \,dx = 0 \]
holds in particular when $\varphi$ merely vanishes at $\partial B_{r}$.  Similarly, 
\[ \int_{B_{r}} D\gamma(\nabla \bar{u})\cdot \nabla\varphi + \lambda \ln R \, \cof \nabla \bar{u} \cdot \nabla \varphi \,dx = 0, \]
so that by subtracting the two and letting 
\[ \Delta = D\gamma(\nabla u) -  D\gamma(\nabla \bar{u})\]
the equation
\[ \int_{B_{r}} \Delta \cdot \nabla \varphi + \lambda \ln R \, \cof(\nabla u - \nabla \bar{u}) \cdot \nabla \varphi \,dx = 0\]
follows.  Using \eqref{kabalev} and \eqref{lam}, the term involving $\cof (\nabla u - \nabla \bar{u}) \cdot \nabla \varphi$ can be rewritten as  
\[\int_{B_{r}} \frac{\lambda J(\partial,_{\tau}(u - \bar{u})) \cdot \varphi}{R} \,dx,\] 
giving
\begin{equation}\label{vivaldi1} \int_{B_{r}} \Delta \cdot \nabla \varphi + \frac{\lambda J(\partial,_{\tau}(u - \bar{u})) \cdot \varphi}{R} \,dx = 0.\end{equation}
Define the function $\eta(t,\eps)$ by
\[ \eta(t,\eps) = \left\{\begin{array}{l l} 0  & \textrm{if} \ 0 \leq t \leq r - \eps \\
\frac{ t - (r  - \eps)}{\eps} & \textrm{if} \ r-\eps \leq t \leq r, \end{array} \right.\]
let $\psi = u - \bar{u}$  and take $\varphi = \eta \psi$ in \eqref{vivaldi1}.   The resulting expression involves in particular the term
\[\int_{B_{r}\setminus B_{r-\eps}} \Delta \cdot \frac{\psi \otimes \er}{\eps} \,dx = 
-\int_{B_{r}\setminus B_{r-\eps}} \Delta \cdot \left((u,_{\scriptscriptstyle{R}}(r,\theta) - \bar{u},_{\scriptscriptstyle{R}}(r, \theta)) \otimes \er  + \frac{l(\eps)}{\eps}\right), \]
where the term $\eps^{-1}l(\eps) \to 0$ as $\eps \to 0$.     Using this, \eqref{vivaldi1} reads
\begin{eqnarray}\label{vivaldi2} & & \int_{B_{r}\setminus B_{r - \eps}} \Delta \cdot \left( \eta (u_{,_{R}} - \bar{u}_{,_{R}}) \otimes \er - (u_{,_{R}}(r,\theta) - \bar{u}_{,_{R}}(r, \theta)) \otimes \er\right) \,dx  + {}\nonumber \\ & & {}+\int_{B_{r}\setminus B_{r - \eps}} \frac{\Delta \cdot l(\eps)}{\eps} \,dx + {}  \nonumber \\
& &  + {}  \int_{B_{r}\setminus B_{r - \eps}} \eta \Delta \cdot (u_{,_{\theta}} - \bar{u}_{,_{\theta}}) \otimes \frac{\et}{R} + \frac{\lambda J(\partial_{\tau}(u - \bar{u})) \cdot \varphi}{R} \,dx = 0. \end{eqnarray}
Since 
\[ \frac{1}{\eps} \int_{B_{r} \setminus B_{r - \eps}} \eta(x) \,dx \to \frac{r}{2} \ \textrm{as} \ \eps \to 0,\]
it is clear from the assumed regularity of $u$ that 
\[\frac{1}{\eps}\int_{B_{r} \setminus B_{r - \eps}} \eta \Delta \cdot (u_{,_{R}} - \bar{u}_{,_{R}}) \otimes \er \,dx \to \frac{1}{2} \int_{\partial B_{r}} \Delta \cdot (u,_{\scriptscriptstyle{R}}(r,\theta) - \bar{u},_{\scriptscriptstyle{R}}(r, \theta)) \otimes \er \,ds.\]
Similarly, 
\[\frac{1}{\eps}\int_{B_{r} \setminus B_{r - \eps}} \eta \Delta \cdot (u,_{\scriptscriptstyle{R}}(r,\theta) - \bar{u},_{\scriptscriptstyle{R}}(r,\theta)) \otimes \er \,dx \to  \int_{\partial B_{r}} \Delta \cdot (u,_{\scriptscriptstyle{R}}(r,\theta) - \bar{u},_{\scriptscriptstyle{R}}(r, \theta)) \otimes \er \,ds.\]
Dividing \eqref{vivaldi2} by $\eps$ and letting $\eps \to 0$, it follows that
\begin{equation}\label{vivaldi3}\int_{\partial B_{r}} \Delta \cdot (u,_{\scriptscriptstyle{R}}(r,\theta) - \bar{u},_{\scriptscriptstyle{R}}(r, \theta)) \otimes \er \,ds = 0.\end{equation}
Now, in view of $u=\bar{u}$ on $\partial B_{r}$,  $\nabla u - \nabla \bar{u} = (u,_{\scriptscriptstyle{R}}(r,\theta) - \bar{u},_{\scriptscriptstyle{R}}(r, \theta)) \otimes \er$ on $\partial B_{r}$, so that \eqref{vivaldi3} becomes 
\begin{equation}\label{vivaldi4}\int_{\partial B_{r}} (D\gamma(\nabla u) - D\gamma(\nabla \bar{u})) \cdot (\nabla u - \nabla \bar{u}) \, ds = 0.\end{equation}
Since $D\gamma$ is strictly monotone, it follows that  
\eqref{vivaldi4} holds only if $\nabla u = \nabla \bar{u}$ on $\partial B_{r}$, concluding the proof.
\end{proof}

A second and more immediate consequence of the strict monotonicity of $D\gamma$ is contained in the next result.  Note that it applies to any two critical points of $G$.

\begin{prop}\label{Bach4}Let $u_{1}$ and $u_{2}$ be critical points of $G$ in $\sca_{p}$.  Then 
\[ \int_{B_{r}} \ln R \, \det ( \nabla u_{1} - \nabla u_{2}) \,dx \leq 0 \]
with equality if and only if $u_{1}=u_{2}$ a.e. in $B_{r}$.\end{prop}
\begin{proof}By definition,  the critical points $u_{1}$ and $u_{2}$ satisfy
\[\int_{B_{r}} D\gamma(\nabla u_{j}) \cdot \nabla \varphi + \lambda \ln R \ \cof \nabla u_{j} \cdot \nabla \varphi \,dx = 0,\]
$j=1,2$, for all smooth test functions $\varphi$ with compact support in the ball $B_{r}$.    Subtracting the two equations and letting $w=u_{1}-u_{2}$ yields
\[ \int_{B_{r}} (D\gamma(\nabla u_{1}) - D\gamma(\nabla u_{2})) \cdot \nabla \varphi + \lambda \ln R \, \cof \nabla w \cdot \nabla \varphi \,dx = 0\]
for all such  $\varphi$.   By an approximation argument we may take $\varphi=w$ in the above.  To be specific, the first term can be approximated by noting that (in view of the assumed $p-$growth of $\gamma$) $D\gamma(\nabla) \in L^{p'}(B_{r})$ whenever $\nabla u \in L^{p}(B_{r})$.  Here, $p'$ is the H\"{o}lder conjugate of $p$.   The second term involving $\ln R$ requires the argument given in the proof of part (a) of Lemma \ref{froebe}.  The result is that
\[ \int_{B_{r}} (D\gamma(\nabla u_{1}) - D\gamma(\nabla u_{2})) \cdot \nabla w + 2 \lambda \ln R \, \det  \nabla w\,dx = 0.\]
The strict  monotonicity of $D\gamma$ together with the fact that $\lambda > 0$ implies
\[ \int_{B_{r}} \ln R \, \det \nabla w \,dx \leq 0\]
with  equality if and only if 
\[(D\gamma(\nabla u_{1}) - D\gamma(\nabla u_{2})) \cdot (\nabla u_{1} - \nabla u_{2}) = 0 \]
a.e. in $B_{r}$, i.e., if and only if $u_{1} = u_{2}$ a.e. in $B_{r}$.
\end{proof}

The following two results will be of use in connection with our discussion of criteria for uniqueness of critical points of $G$.  The first is essentially an identity; the second is a technical lemma whose function will become apparent in the course of the proof of Proposition \ref{bachprop3} below.

\begin{lemma}\label{bach2b} Let $u \in \sca_{p}$ and suppose $\bar{u} = a R e_{\scriptscriptstyle{R}}(k\theta)$ is a one-homogeneous critical point of $G$.   Then
\begin{equation}\label{ix}
\int_{B_{r}} \ln R \, \cof \nabla u \cdot \nabla \bar{u} \,dx = 2\pi k a^{2}r^{2}\ln r -  ak \int_{B_{r}} u \cdot e_{\scriptscriptstyle{R}}(k\theta)\,d\theta\,dR.
\end{equation}
\end{lemma}
\begin{proof}Using the form of gradient given in \eqref{ravelI}, it is straightforward to check that 
\[ \cof \nabla u \cdot \nabla \bar{u} = Ju_{,_{\scriptscriptstyle{R}}} \cdot \bar{u}_{,_{\tau}} - Ju_{,_{\tau}} \cdot \bar{u}_{,_{\scriptscriptstyle{R}}}.\]
Inserting $\bar{u}=aRe_{\scriptscriptstyle{R}}(k\theta)$ and integrating gives
\begin{eqnarray*}\int_{B_{r}}\ln R \, \cof \nabla u \cdot \nabla \bar{u}\,dx & = & \int_{B_{r}} R\ln R \,Ju_{,_{\scriptscriptstyle{R}}} \cdot ae_{\scriptscriptstyle{R}}(k\theta) \,dR \,d\theta +{} \\
&  & {}  -ak\int_{B_{r}} \ln R \, Ju\cdot \et(k\theta)\,dR\,d\theta \\ 
& = & \int_{0}^{2\pi} \left[R \ln R Ju \cdot a\et(k \theta)\right]_{R=0}^{R=r}\,d\theta +{} \\
&  & {}  -ak\int_{B_{r}}(1+ \ln R) \, Ju\cdot \et(k\theta)\,dR\,d\theta +{} \\
&  & {}  + ak\int_{B_{r}} \ln R \, Ju\cdot \et(k\theta)\,dR\,d\theta \\
& = &  k\int_{0}^{2\pi} r^{2} \ln r \, a\et(k\theta) \cdot a\et(k\theta) \,d\theta {}+ \\
&  & {}+ ak \int_{B_{r}}u \cdot J\et(k\theta)\,dR\,d\theta \\
& = & 2 \pi k a^{2} r^{2} \ln r  -  ak \int_{B_{r}} u \cdot e_{\scriptscriptstyle{R}}(k \theta) \,dR \,d \theta.
\end{eqnarray*}
\end{proof}

\begin{lemma}\label{bach2a} Let $u$ and $\bar{u}$ be critical points of $G$ in $\sca_{p}$ and suppose that either
\begin{itemize}\item[(C1)] $\det \nabla u = \det \nabla \bar{u}$ a.e. in  $B_{r}$, or
\item[(C2)] $G(u)=G(\bar{u})$, $u$ is $C^{1}$ in a neighbourhood of $\partial B_{r}$ and $\gamma$ is homogeneous of degree $p$,  $p \neq 2$.  
\end{itemize}
Then
\begin{equation}\label{scheide}\int_{B_{r}} \ln R \, \det \nabla u \,dx = \int_{B_{r}} \ln R \, \det \nabla \bar{u} \,dx, \end{equation} 
and their common value is
\begin{equation}\label{8stringmandolin} \int_{B_{r}} \ln R \, \det \nabla \bar{u} \,dx = \pi  k  a^{2}  (2 r^{2} \ln r - r^2).\end{equation}
\end{lemma}
\begin{proof}If condition (C1) holds then \eqref{scheide} is immediate.   Therefore  assume that (C2) holds.     A version of Green's Theorem implies that 
{\small\[ \int_{B_{r}} (D\gamma(\nabla u) + \lambda \ln R\, \cof \nabla u) \cdot \nabla \bar{u} \,dx  = r \int_{0}^{2\pi} (D\gamma(\nabla u) + \lambda \ln r \, \cof \nabla u ) \cdot \bar{u} \otimes \er(\theta) \,d\theta,\]}
and similarly 
{\small\[ \int_{B_{r}} (D\gamma(\nabla \bar{u}) + \lambda \ln R\, \cof \nabla \bar{u}) \cdot \nabla u \,dx  = r \int_{0}^{2\pi} (D\gamma(\nabla \bar{u}) + \lambda \ln r \, \cof \nabla \bar{u}) \cdot \bar{u} \otimes \er(\theta) \,d\theta.\]}   

By Lemma \ref{regina}, we may assume that $\nabla u = \nabla \bar{u}$ on $\partial B_{r}$, so that the right-hand sides of the last two equations are equal.   
In particular, we can then take $u=\bar{u}$ in each  (since they hold as identities) and conclude that 
\[\int_{B_{r}} D\gamma(\nabla u)\cdot \nabla u + 2 \lambda \ln R\, \det \nabla u \,dx =    \int_{B_{r}} D\gamma(\nabla \bar{u})\cdot \nabla \bar{u} + 2\lambda \ln R\, \det \nabla \bar{u} \,dx.\]
Since $\gamma$ is assumed to be homogeneous of degree $p$,  it follows that
\[ D\gamma(\nabla u)  \cdot \nabla u = p\gamma(\nabla u),\]
and similarly for $\bar{u}$.  Hence
\[  \int_{B_{r}}\frac{p}{2}\gamma(\nabla u) + \lambda \ln R \, \det \nabla u \,dx = 
\int_{B_{r}}\frac{p}{2}\gamma(\nabla \bar{u}) + \lambda \ln R \, \det \nabla \bar{u} \,dx.\] 
Now we apply the hypothesis that  $G(u) = G(\bar{u})$,  which gives
\[\int_{B_{r}} \gamma(\nabla u) + \lambda \ln R \, \det \nabla u \,dx = 
\int_{B_{r}}\gamma(\nabla \bar{u}) + \lambda \ln R \, \det \nabla \bar{u} \,dx.\] 
Subtracting the two  equations and using the assumption $p \neq 2$ implies that
\[ \int_{B_{r}} \gamma(\nabla u) \,dx = \int_{B_{r}} \gamma(\nabla \bar{u}) \,dx,\]
which, since $G(u) = G(\bar{u})$, immediately implies equation \eqref{scheide}.   Finally,  it can be checked that the map $\bar{u}$ has constant Jacobian $a^{2}k$. Hence equation \eqref{8stringmandolin}.
\end{proof}

\begin{rem}Note that the proof of \eqref{scheide} under assumption $(C2)$ also implies that $D\gamma$ is `self-adjoint' on  critical points of $G$ in the sense that
\[ \int_{B_{r}} D\gamma(\nabla u) \cdot \nabla \bar{u} \,dx = \int_{B_{r}} D\gamma(\nabla \bar{u}) \cdot \nabla u \,dx.\]
This holds even when $\gamma$ is not assumed to be homogeneous of degree $p$.
\end{rem}

\begin{rem}  Under additional regularity assumptions, \cite{KS84,TaII} have studied the uniqueness problem for critical points of functionals involving a Lagrange multiplier.   The singularity associated with $\ln R$ prevents a direct application of their results to the functional $G$, but perhaps their methods could be adapted to work in this case.  
\end{rem}

We now give a uniqueness criterion for critical  points of $G$ in  $\sca_{p}$.

\begin{prop}\label{bachprop3}  Let $u$ and $\bar{u}$ be critical  points of $G$ and suppose that either (C1) or (C2) holds.   Then  
\begin{equation}\label{lute} \int_{B_{r}} u \cdot \er(k \theta) \,dR \,d\theta \leq \pi a r^{2}\end{equation}
with equality if and only if $u=\bar{u}$ a.e. in  $B_{r}$. 
\end{prop}
\begin{proof} The aim is to apply Proposition \ref{Bach4}. Let $w=u-\bar{u}$ and calculate 
\begin{eqnarray*} \int_{B_{r}} \ln R \,\det \nabla w \,dx &  = &   \int_{B_{r}} \ln R \, \det \nabla u  + \ln R \, \det \nabla \bar{u}\,dx  +{} \\
&  &{} - \int_{B_{r}}\ln R \,\cof \nabla u \cdot \nabla \bar{u}\,dx \\
&  = & 2 \int_{B_{r}} \lambda \ln R \det \nabla \bar{u}\,dx + {} \\
&  & {} -\left(2\pi a^2r^{2}\ln r - ak  \int_{B_{r}} u \cdot \er(k\theta) \,dR \,  d\theta\right) \\
&=& \pi  k  a^{2}  (2 r^{2} \ln r - r^2) - 2\pi a^2r^{2}\ln r + {} \\
& & {}+ ak  \int_{B_{r}} u \cdot \er(k\theta) \,dR \,  d\theta\\
& = &  ak\left (\int_{B_{r}} u \cdot \er(k\theta) \,dR \,  d\theta - \pi  a r^2\right).\end{eqnarray*}
Here, Lemma \ref{bach2a} has been  used to replace $\int_{B_{r}}\ln R \,\det \nabla u \,dx$ with \newline $\int_{B_{r}}\ln R \,\det \nabla \bar{u} \,dx$, and the identity in  Lemma \ref{bach2b} has been applied to the term involving $\cof \nabla u \cdot \nabla \bar{u}$.
In summary,
\begin{equation}\label{CPE}\int_{B_{r}} \ln R\, \det \nabla w \, dx=  ak\left (\int_{B_{r}} u \cdot \er(k\theta) \,dR \,  d\theta - \pi  a r^2\right).\end{equation}
By Proposition \ref{Bach4}, the right-hand side of \eqref{CPE} satisfies
\[ak\left (\int_{B_{r}} u \cdot \er(k\theta) \,dR \,  d\theta - \pi  a r^2\right) \leq 0\]
with equality if and only if $u = \bar{u}$ a.e. in  $B_{r}$.   Since $ak > 0$, we can  now deduce the conclusion of the proposition.
\end{proof}

Note that the right-hand side of \eqref{lute} is $\int_{B_{R}} \bar{u} \cdot \er(k\theta) \,dR  \,d \theta$.  Therefore one interpretation of \eqref{lute} and Proposition \ref{bachprop3} is that among all possible critical points of $G$ in $\sca_{p}$ satisfying either (C1) or (C2) it is uniquely $\bar{u}$ which maximizes $\int_{B_{r}} u \cdot \er(k\theta)\,dR\,d\theta$.

\section{Critical points as minimizers}\label{calcvar}
 
Functionals of the form
\[G(u) = \int_{B_{r}} \gamma(\nabla u) + \lambda(x) \det \nabla u\,dx\]
are typically associated with variational problems involving constraints on the Jacobian $\det \nabla u$.  In such cases the function $\lambda(x)$ is a Lagrange multiplier, or pressure; it is not necessarily explicitly known \emph{a priori}, though it happens to be in our case, where $\lambda(x)$ is proportional to $\ln |x|$.   The fact that $\bar{u}$ constructed in Proposition \ref{freewill}  solves the Euler-Lagrange equation 
\begin{equation}\label{el5}\int_{B_{r}} D\gamma(\nabla \bar{u}) \cdot \nabla \varphi + \lambda \ln R \, \cof \nabla \bar{u} \cdot \nabla \varphi \,dx = 0 \ \ \forall \varphi \in C_{c}^{\infty}(B_{r},\mathbb{R}^{2})\end{equation}
and satisfies $\det \nabla \bar{u} = ak^{2}$ a.e. in $B_{r}$, i.e., is constant almost everywhere, suggests that the stationarity condition \eqref{el5} may well have a variational origin.  Could it be that $\bar{u}$ minimizes $G$ among maps in $\sca_{p}$ satisfying $\det \nabla u = ak^{2}$ almost everywhere?    We discuss this question below in Section \ref{con}.   

We point out in Section \ref{uncon} that $\bar{u}$ is in fact the global minimizer of the functional
\[ E(u)=\int_{B_{r}} k^{2}|u_{,_{\scriptscriptstyle{R}}}|^{2} + |u_{,_{\tau}}|^{2}\,dx\] 
among \emph{all} functions in $\sca_{2}$.   There are some remarkable similarities between this and an example given by Meyers in his work \cite{Meyers63} on reverse H\"{o}lder inequalities in elliptic regularity theory.  See the discussion following 
Proposition \ref{oomin} for details.

\subsection{The mapping $\bar{u}$ as a constrained minimizer}\label{con}

Let
\[\sca_{p}'=\{ u  \in W^{1,p}(B_{r},\mathbb{R}^{2}): \ \det \nabla u = 1   \ \textrm{a.e.}, \ u = \bar{u} \ \textrm{on} \ \partial B_{r}\}.\]
Here, we will choose the coefficient $a$ in $\bar{u} = aR \er (k \theta)$ so that $ak^{2} = 1$.  Hence $\bar{u}$ is also a member of $\sca_{p}'$.   When considered as a  functional on $\sca_{p}'$ it is clear that $G(u)$ differs from $\int_{B_{r}} \gamma(\nabla u)$ only by the constant term $\int_{B_{r}}\lambda \ln R \,dx$.   Therefore for the rest of this section we take
\[ G(u) =  \int_{B_{r}} \gamma(\nabla u) \,dx,\] 
where $\gamma$ will  be assumed to satisfy the main hypotheses leading to the result of  Lemma \ref{bach0}, as well as the $p-$growth hypothesis
\[ c_{1}|F|^{p} \leq \gamma(F) \leq c_{2}|F|^{p} + c_{3} \ \ \forall F \in  \mathbb{R}^{2 \times 2}\]
for positive constants $c_{1},c_{2}$ and $c_{3}$.  We assume that $p > 2$ for reasons explained later.

In  Theorem \ref{hybridlarch} below we show that $\bar{u}$ minimizes $G$ among all functions $u$ in $\sca_{p}'$ for which the set $u(B_{r})$ has the same `periodicity' as $\bar{u}$.  The notion of periodicity is made precise in  Lemma \ref{thomas}.   The technical methods we use are an adaptation of those presented in Sivaloganathan and Spector \cite{SSI,SSII};  we include details of proofs only to  keep  the paper self-contained and to avoid confusion with the scalings involved.  In  practice, if the reader is already familiar with \cite{SSI} then he or she should be able to deduce Theorem \ref{hybridlarch} from Lemma \ref{thomas} and \cite[Section 5]{SSI}.

We conjecture that $\bar{u}$ is in fact the global minimizer of $G$ in the full class $\sca_{p}'$, that is, among functions where no topological control is applied to the image set $u(B_{r})$.  The  techniques of \cite{SSI} do not seem to work in this case.

In keeping with  the notation  introduced in \cite{SSI,SSII}, let $C_{\scriptscriptstyle{R}}$ be the circle of radius $R$ and centre zero in $\mathbb{R}^{2}$.  

\begin{lemma}\label{thomas} Let  $\phi  \in W^{1,p}(B_{r};\mathbb{R}^{2})$ satisfy  $\phi(x) =  x$ for all  $x \in \partial B_{r}$ and $\det \nabla\phi = 1$ a.e. in $B_{r}$. 
Define
\[ \phi^{(k)}(R,\theta) = \phi\left(k^{-\frac{1}{2}}R,k\theta\right)\]
for all  $(R,\theta)$.   Then $\phi^{(k)} \in \sca_{p}'$,  and the image set satisfies
\[ \phi^{(k)}(B_{r}) \bigcup_{0<R\leq r}\bigcup_{0 \leq j \leq k-1}\phi\left(C_{k^{-\frac{1}{2}}\scriptscriptstyle{R}}\right).\]
That is, each set $\phi^{(k)}(C_{\scriptscriptstyle{R}})$ is covered $k$ times by the set  $\phi\left(C_{k^{-\frac{1}{2}\scriptscriptstyle{R}}}\right)$.
\end{lemma}
\begin{proof}A calculation shows that
\begin{eqnarray*}\det \nabla \phi^{(k)} &  = & \frac{1}{k^{\frac{1}{2}}R}J\phi,_{\scriptscriptstyle{R}}(k^{-\frac{1}{2}}R,k\theta) \cdot k \phi,_{\theta}(k^{-\frac{1}{2}}R,k\theta)\\
& = & J\phi,_{\scriptscriptstyle{R}}(k^{-\frac{1}{2}}R,k\theta) \cdot \frac{\phi,_{\theta}(k^{-\frac{1}{2}}R,k\theta)}{k^{-\frac{1}{2}}R}\\
&  = & \det \nabla \phi(\rho,\sigma)
\end{eqnarray*}
where $\rho=k^{-\frac{1}{2}}R$ and $\sigma = k\theta$.   Thus $\det \nabla \phi^{(k)} = 1$ a.e. in $B_{r}$.  Note also that, by definition, $\phi^{(k)}(r,\theta)=\phi(k^{-\frac{1}{2}}r,k\theta)$, so that on applying the hypothesis that $\phi$ is the identity on $\partial B_{r}$, it must be that $\phi^{(k)}(r,\theta) = k^{-\frac{1}{2}}r\er(k\theta)$, i.e., $\phi^{(k)}=\bar{u}$ on $\partial  B_{r}$.   Hence $\phi^{(k)}$ belongs to the class $\sca_{p}'$.

The last assertion of the lemma follows by noting that $\phi^{(k)}$ is by definition $\frac{2\pi}{k}-$periodic in the angular variable.  Therefore the images $\phi^{(k)}(C_{\scriptscriptstyle{R}})$ are formed by overlaying $k$ copies of the sets $\phi\left(C_{k^{-\frac{1}{2}}\scriptscriptstyle{R}}\right)$.

\end{proof}

Now let
\[\mathcal{D} =  \{\phi  \in W^{1,p}(B_{r};\mathbb{R}^{2}): \   \phi(x) =  x \ \forall x \in \partial B_{r}, \ \det \nabla\phi = 1 \  \textrm{a.e. in} \ B_{r}\}\] 
and define
\[A_{p}''=\{ u \in \sca_{p}': \ u=\phi^{(k)}, \ \phi \in \mathcal{D}\}.\]
By Lemma \ref{thomas},  $\sca_{p}'' \subset \sca_{p}'$, and note that, by inspection, $\bar{u}$ belongs to $\sca_{p}''$.

We show in the next two results that $\bar{u}$ is the global minimizer of $G$ among maps in the class $\sca_{p}''$.

\begin{lemma}\label{westernlarch} Let $u = \phi^{(k)} \in \sca_{p}''$.  Then for a.e. $R \in (0,r)$,
\[\int_{C_{\scriptscriptstyle{R}}}  \left| \phi^{(k)},_{\tau} \right| \,d\sch^{1} \geq 2\pi k^{\frac{1}{2}}R.\]
\end{lemma}
\begin{proof} By a change of variables, we see that for all  $R \in  (0,r)$
\begin{equation}\label{period}\int_{C_{\scriptscriptstyle{R}}}  \left| \phi^{(k)},_{\tau} \right| \,d\sch^{1}   = k \int_{C_{\scriptscriptstyle{k^{-\frac{1}{2}}R}}} |\phi,_{\tau}|\,d\sch^{1}.\end{equation}
Let $\rho = k^{-\frac{1}{2}}R$.  Since $\phi \in W^{1,p}(B_{r})$, it is known that $\phi\arrowvert_{C_{\rho}} \in  W^{1,p}(C_{\rho})$ for a.e. $\rho \in  (0,r)$.  (See \cite[Theorem 5.16]{FG95},  for example.)  Thus by \cite{MM73}, or see \cite[Theorem 1]{Sv88}, 
\begin{equation}\label{mm73}  \int_{C_{\scriptscriptstyle{k^{-\frac{1}{2}}R}}}|\phi,_{\tau}|\,d\sch^{1} \geq \sch^{1}(C_{\rho})\end{equation}
for a.e. $\rho \in (0,k^{-\frac{1}{2}}r)$.   By  \cite[Part 2(a), Lemma 3.5]{MS95}, the measure-theoretic boundary of the topological image $\phi^{\textrm{top}}(B_{\rho})$ of $B_{\rho}$ under $\phi$ differs from $\phi(B_{\rho})$ by at most a set of $\sch^{1}$ measure zero.   We recall that $\phi^{\textrm{top}}(B_{\rho})$ is defined in terms of the Brouwer degree by 
\[ \phi^{\textrm{top}}(B_{\rho}) =  \{y \in  \mathbb{R}^{2} \setminus \phi(C_{\rho}): \ d(\phi,B_{\rho},y) \neq 0.\}\]
(See e.g. \cite{FG95} for details of the Brouwer degree.)
Thus
\begin{equation}\label{reduced} \sch^{1}(C_{\rho})  = \sch^{1}(\partial^{\ast}(\phi^{\textrm{top}}(B_{\rho})).\end{equation}
Following \cite{SSI}, the isoperimetric inequality in the plane gives
\begin{equation}\label{isoperim} \sch^{1}(\partial^{\ast}(\phi^{\textrm{top}}(B_{\rho}))) \geq 
2 \pi^{\frac{1}{2}}\left(\mathcal{L}^{2}(\phi^{\textrm{top}}(B_{\rho}))\right)^{\frac{1}{2}}.\end{equation}
If we can show that
\begin{equation}\label{equalarea}\mathcal{L}^{2}(\phi^{\textrm{top}}(B_{\rho})) = \mathcal{L}^{2}(\phi(B_{\rho}))\end{equation}
then the right-hand side of \eqref{isoperim} would become $2\pi \rho$.  Inserting this into \eqref{period}-\eqref{isoperim} would give the desired lower bound $2\pi k^{\frac{1}{2}}R$.   

To prove \eqref{equalarea}, we note that because $p > 2$ and because $\phi$ coincides with a homeomorphism on  $\partial B_{r}$,  by  \cite[Theorem  1]{Ba80} $\phi$ is continuous on  $\bar{B_{r}}$ and invertible a.e. in $B_{r}$.   Then \cite[Proposition 5.5(f)]{SSI}   applies to $\phi$, so  that $d(\phi,B_{\rho},y) =  1$ or $0$ for all  $y \notin \phi(C_{\rho})$ for each $\rho  \in (0,k^{-\frac{1}{2}}R)$.  We also see that the hypotheses of \cite[Theorem 5.30]{FG95} are satisfied, so  that  in  particular
\[ \int_{B_{\rho}} \det \nabla \phi \,dx = \int_{\mathbb{R}^{2}} d(\phi,B_{\rho},y) \,dy.\]   
The left-hand side is easily evaluated by setting $\det \nabla \phi = 1$ a.e., giving $\pi \rho^{2}$.   We claim that the right-hand side gives $\scl^{2}(\phi(B_{\rho}))$.  Firstly, we show that
$\phi(B_{\rho}) \setminus \phi^{\textrm{top}}(B_{\rho})$ is a set of $\scl^{2}-$measure zero.  (Note that $\phi^{\textrm{top}}(B_{\rho}) \subset \phi(B_{\rho})$ using the definition together with basic properties of the degree.)    Let $y \in \phi(B_{\rho}) \setminus \phi^{\textrm{top}}(B_{\rho})$.  We can assume that $y \notin \phi(\partial B_{\rho})$ since this set is null (by \cite[Theorem 5.32]{FG95}).  Then there is $x_{1} \in B_{\rho}$ such that $\phi(x_{1}) = y$, and, since $y$ is not in the topological image of $\phi$, $d(\phi,B_{\rho},y)=0$.  But $\phi$ agrees with the identity on $\partial B_{r}$, so $d(\phi,B_{r},y) = 1$ for all $y \in B_{r}$.   Using \cite[Chapter 2]{FG95}, for all $y \notin \phi(C_{\rho})$
\begin{eqnarray*} d(\phi,B_{r},y) &  = &  d(\phi,B_{r}\setminus C_{\rho},y)\\
& = & d(\phi,B_{\rho},y) + d(\phi, B_{r}\setminus \bar{B_{\rho}},y).\end{eqnarray*}
It follows that  $d(\phi, B_{r}\setminus \bar{B}_{\rho},y) = 1$, and so there must be $x_{2} \in B_{r}\setminus \bar{B_{\rho}}$ such that $y=\phi(x_{2})$.   Therefore $\phi$ is not $1-1$ at $y$, and since $\phi$ is a.e. $1-1$ it must be that $y$ belongs to a null set.   Hence,  since $d(\phi,B_{\rho},y)=1$ for all $y \in  \phi^{\textrm{top}}(B_{\rho})$, we have
\[d(\phi, B_{\rho}, y)  = \chi_{\phi(B_{\rho})}(y) \ \textrm{a.e.} \ y  \in  B_{r},\]
so  that 
\[ \int_{\mathbb{R}^{2}} d(\phi, B_{\rho}, y) \,dy = \mathcal{L}^{2}(\phi(B_{\rho})).\]
\end{proof}

We now apply Lemma \ref{westernlarch} to the functional $G$,  following the method of \cite[Remark 3.10]{SSI}

\begin{thm}\label{hybridlarch} Let $u \in \sca_{p}''$.  Then $G(u) \geq G(\bar{u})$.\end{thm}
\begin{proof}
Let $u=\phi^{(k)}$ for some $\phi \in \mathcal{D}$.  Note that since $\det \nabla \phi^{(k)} = 1$ a.e. in $B_{r}$, it follows from \eqref{det} and the Cauchy-Schwarz inequality that  
\begin{equation}\label{bminormass}|\phi^{(k)},_{\scriptscriptstyle{R}} |\geq \frac{1}{|\phi^{(k)},_{\tau}|}\end{equation}
a.e. in  $B_{r}$.  Following \cite{SSI}, we  first  show
\[\int_{B_{r}}|\nabla u| \,dx \geq \int_{B_{r}}|\nabla \bar{u}|\,dx.\]
To  that end,  let
\[h(s) =  (s^2+s^{-2})^{\frac{1}{2}}\]
for all $s>0$.  Note that $h$ is convex on $(0,\infty)$ and strictly increasing on $[1,\infty)$.
By the coarea formula (\cite[Theorem 1, Section 3.4.2]{EG}), \eqref{bminormass}, and Jensen's inequality,
\begin{eqnarray}\nonumber \int_{B_{r}}|\nabla u|\,dx & = & \int_{0}^{r} \int_{C_{\scriptscriptstyle{R}}}\left(|\phi^{(k)},_{\scriptscriptstyle{R}}|^{2} +  |\phi^{(k)},_{\tau}|^{2}\right)^{\frac{1}{2}} \,d\sch^{1} \,dR \\
\nonumber & \geq & \int_{0}^{r} \int_{C_{\scriptscriptstyle{R}}}\left(|\phi^{(k)},_{\tau}|^{-2} +  |\phi^{(k)},_{\tau}|^{2}\right)^{\frac{1}{2}} \,d\sch^{1} \,dR \\
\nonumber & = & \int_{0}^{r} \int_{C_{\scriptscriptstyle{R}}}h\left(|\phi^{(k)},_{\tau}|\right) \,d\sch^{1} \,dR \\
\label{getty}&  \geq &  \int_{0}^{r} 2 \pi  R\, h\left(\dashint_{C_{\scriptscriptstyle{R}}}|\phi^{(k)},_{\tau}| \,d\sch^{1}\right) \,dR.
\end{eqnarray}
By  Lemma \ref{westernlarch}, 
\[ \dashint_{C_{\scriptscriptstyle{R}}}|\phi^{(k)},_{\tau}| \,d\sch^{1} \geq  k^{\frac{1}{2}}\]
for a.e. $R \in  (0,r)$.  Since $k$ is a positive integer,  it follows that
\[h\left(\dashint_{C_{\scriptscriptstyle{R}}}|\phi^{(k)},_{\tau}| \,d\sch^{1}\right) \geq h(k^{\frac{1}{2}}).\]
Inserting this into \eqref{getty} and  using the fact that $\bar{u}=k^{-\frac{1}{2}}R\er(k\theta)$ satisfies
\[|\nabla \bar{u}| = h(k^{\frac{1}{2}}),\]
we conclude  that 
\begin{equation}\label{lastone}\int_{B_{r}}|\nabla u| \,dx  \geq \int_{B_{r}}|\nabla \bar{u}|\,dx.\end{equation}
Finally, since $\gamma(F)=f(|F|)$, where $f$ is convex and increasing, Jensen's inequality combined  with  \eqref{lastone} gives
\begin{eqnarray*}\int_{B_{r}} \gamma(\nabla u) \,dx & = & \int_{B_{r}}f(|\nabla u|)\,dx \\
&  \geq & \pi r^{2} f\left(\dashint_{B_{r}} |\nabla u|\,dx \right) \\
&  \geq & \pi r^{2} f\left(\dashint_{B_{r}} |\nabla \bar{u}|\,dx \right) \\
&  =  & \pi r^{2} f(h(k^{\frac{1}{2}})) \\
&  = &  \int_{B_{r}} \gamma(\nabla \bar{u}) \,dx.\end{eqnarray*}
\end{proof}

\begin{rem}Note that the proofs above  also show that $\bar{u}$ is the unique minimizer of $G$ in $\sca_{p}''$.   This can easily be seen by examining the conditions for equality in inequalities \eqref{isoperim},\eqref{bminormass} and in the application \eqref{getty} of Jensen's inequality.\end{rem}

\subsection{The mapping $\bar{u}$ as a global minimizer}\label{uncon}

Define the functional $E$ by
\begin{equation}\label{E}E(u)=\int_{B_{r}} k^{2}|u_{,_{\scriptscriptstyle{R}}}|^{2} + |u_{,_{\tau}}|^{2}\,dx
\end{equation} 
and note that $E$ is well defined on the class $\sca_{2}$.  Let $\bar{u}$ be as in  Section \ref{con}. 

\begin{prop}\label{oomin}The map $\bar{u}$ solves the Euler-Lagrange equation 
\begin{equation} \label{smetena} 
\int_{B_{r}} u_{,_{\tau}} \cdot \varphi_{,_{\tau}} + k^{2} u_{,_{\scriptscriptstyle{R}}} \cdot \varphi_{,_{\scriptscriptstyle{R}}}\,dx = 0 \ \ \forall \varphi \in C_{c}^{\infty}(B_{r},\mathbb{R}^{2})
\end{equation}
associated with the functional $E$.   Consequently, $\bar{u}$ is the unique global minimizer of $E$ in the class $\sca_{2}$.
\end{prop}
\begin{proof}By definition, the Euler-Lagrange equation for $\bar{u}$ is
\[ \partial_{\eps}\arrowvert_{\eps = 0} E(\bar{u} + \eps \varphi) = 0,\]
where $\varphi$ is a smooth test function with compact support in $B_{r}$. 
A standard argument yields \eqref{smetena}.   To see that $\bar{u}$ solves it we must have
\[\int_{B_{r}} k^{\frac{1}{2}}\et(k\theta) \cdot \varphi_{,_{\tau}} + k^{\frac{3}{2}} 
\er(k \theta) \cdot \varphi_{,_{\scriptscriptstyle{R}}}\,R \,dR d \theta  = 0\]
for all such $\varphi$.  Integrating by parts with respect to $\theta$ in the first term and with respect to $R$ in the second, we require
\[\int_{B_{r}} k^{\frac{3}{2}}\er(k\theta) \cdot \varphi \,d\theta \,dR + k^{\frac{3}{2}}\int_{0}^{2\pi}\left([R\varphi \cdot \er(k\theta)]_{R=0}^{R=1} - \int_{0}^{1} \varphi \cdot \er(k\theta) \,dR\right) \,d\theta = 0.\]
Since $\varphi$ has compact support in $B_{r}$ and is bounded at $0$, the second term vanishes.   The remaining terms cancel.

To see that $\bar{u}$ is the unique global minimizer of $E$ in $\sca_{2}$ note that by \eqref{smetena} and a simple approximation, 
\[ E(v) =  E(\bar{u})  + E(v - \bar{u})\]
holds for any $v \in \sca_{2}$. 
Thus $E(v) \geq E(\bar{u})$ with  equality if and only if $E(v-\bar{u})=0$, i.e. if and only if $v=\bar{u}$ a.e. in $B_{r}$.   
\end{proof}

At  first sight it appears that Proposition \ref{oomin} provides us with a non-smooth minimizer of a strongly uniformly elliptic functional $E$.   But when the integrand 
\[ W(u_{,_{\scriptscriptstyle{R}}},u_{,_{\tau}}) =  k^{2}|u_{,_{\scriptscriptstyle{R}}}|^{2} + |u_{,_{\tau}}|^{2}\]
appearing in $E(u)$ is expressed in terms of the derivatives $u_{,_{x_{1}}}$ and $u_{,_{x_{2}}}$ we see that
\[W(u_{,_{\scriptscriptstyle{R}}},u_{,_{\tau}})  = \widetilde{W}(x,u_{,_{x_{1}}}, u_{,_{x_{2}}}),\]
where
\[\widetilde{W}(x,u_{,_{x_{1}}}, u_{,_{x_{2}}}) = l_{1}(x)|u_{,_{x_{1}}}|^{2} + l_{2}(x) u_{,_{x_{1}}} \cdot u_{,_{x_{2}}} + l_{3}(x)|u_{,_{x_{2}}}|^{2}\]
and
\begin{eqnarray*}
l_{1}(x) & = & (k^{2} - 1)\cos^{2}\theta +1 \\
l_{2}(x) &  =  & 2(k^2 - 1) \sin \theta \cos \theta \\
l_{3}(x) & = & (k^{2} - 1) \sin^{2}\theta + 1.
\end{eqnarray*}
In particular, the coefficients $l_{j}$ belong to $L^{\infty}(B_{r})\setminus C^{0}(B_{r})$, so that even though $\widetilde{W}$ is strongly uniformly elliptic in the sense that
\[ |\Pi|^{2} \leq \frac{\partial^{2}\widetilde{W}(x,F)}{\partial F_{ij} \partial F_{st}}\Pi_{ij}\Pi_{st} \leq k^{2} |\Pi^{2}| \ \ \forall F, \Pi \in \mathbb{R}^{2 \times 2}, \ \forall x \in B_{r},\]
classical regularity theory (see e.g. \cite[Chapter III, Theorem 3.1]{Gi83}) predicts only that weak solutions of \eqref{smetena} should, depending on the setting, belong to $C^{0,\mu}(B_{r})$ for some $0 <  \mu < 1$.  Crucially, we do not expect higher regularity of weak solutions when the coefficients are not continuous, as they plainly are not in this case.  Note that $\bar{u}$ is Lipschitz continuous, and so lies in  $C^{0, \mu}(B_{r})$ for every $\mu \in  (0,1)$.  Thus the example here is consistent with established theory.    We also remark that this example is consistent with Morrey's more powerful result \cite[Theorem 4.3.1]{Mo}.  (Or see \cite[Chapter V, Corollary 3.1]{Gi83}.)

Finally, we remark that the system \eqref{smetena} and its solution $\bar{u}$ bear a striking resemblance to one constructed by Meyers in \cite{Meyers63}.  His example, which is discussed below, showed that the improved regularity of elliptic systems with merely $L^{\infty}$ coefficients is controlled by the `strength of ellipticity' of the system.   We refer the reader to \cite{Meyers63} or \cite[Chapter V, Theorem 2.5]{Gi83} for further  details.  Meyers considered the following scalar equation:
\begin{equation}\label{arpeggiata} (au_{,_{x_{1}}}+bu_{,_{x_{2}}})_{,_{x_{1}}} + (bu_{,_{x_{1}}}+cu_{,_{x_{2}}})_{,_{x_{2}}} = 0
\end{equation} 
for $u: \mathbb{R}^2 \to \mathbb{R}$,
and where
\begin{eqnarray*} a(x) & = &  \cos^{2}\theta + \mu^{2} \sin^{2}\theta \\
b(x) & = &  (1-\mu^2)\sin \theta \cos \theta \\
c(x) & = & \sin^{2}\theta + \mu^{2} \cos^{2}\theta. 
\end{eqnarray*}
Here, $\mu \in  (0,1)$ is a constant.   If we let $u=(u_{1}, u_{2})^{T}$ be such that both $u_{1}$ and $u_{2}$ solve \eqref{arpeggiata} then  the weak form of the resulting system is, on $B_{1}$, say, 
\begin{equation}\label{tarantella}
\int_{B_{1}}\nabla u ^{T} A \cdot \nabla \varphi \,dx = 0 \ \ \ \forall \varphi \in  C_{c}^{\infty}(B_{1}; \mathbb{R}^{2}).
\end{equation}
The $2 \times 2$ matrix $A$ is given by 
\begin{displaymath} A = \left(\begin{array}{l l} 1 & 0 \\ 0 & \mu^2\end{array}\right)
\end{displaymath}
in the basis $\mathcal{B} = \{\er \otimes \er,\ \er \otimes \et, \ \et \otimes \er, \ \et  \otimes \et\}$.  It can be checked that $u_{\mu}=R^{\mu}\er(\theta)$ solves \eqref{tarantella} for $0 < \mu  \leq 1$.   Now compare the weak form of \eqref{smetena}:  in the same coordinate system it is 
\[\int_{B_{r}} \nabla u  \bar{A} \cdot \nabla \varphi \,dx = 0 \ \ \ \forall \varphi \in  C_{c}^{\infty}(B_{1}; \mathbb{R}^{2}),\]
where the $2 \times 2$ matrix $\bar{A}$ is given by 
\begin{displaymath} \bar{A} = \left(\begin{array}{l l} k^{2} & 0 \\ 0 & 1 \end{array}\right).
\end{displaymath}
We can see that Meyers used the same mechanism to ensure that $a,b,c$ belong to $L^{\infty}(B_{1})\setminus C^{0}(B_{1})$.

\end{document}